\documentclass{article}

\usepackage[latin1, utf8]{inputenc}
\usepackage[T1]{fontenc}
\usepackage[francais,english]{babel}
\usepackage{hyperref}
\usepackage{amsmath}
\usepackage{amssymb}
\usepackage{amsthm}
\usepackage{xkeyval}
\usepackage[left=2.5cm,top=3cm,right=2.5cm,bottom=3cm,bindingoffset=0.5cm]{geometry}
\usepackage{xcolor}
\newcommand{\RR}{\mathbb{R}}
\newcommand{\NN}{\mathbb{N}}
\newcommand{\Z}{\mathbb{Z}}
\newcommand{\CC}{\mathbb{C}}
\newcommand{\norm}{\|\cdot\|}

\newcommand{\dist}{\mathcal{E}^\prime}

\newcommand{\vertiii}[1]{{\vert\kern-0.25ex\vert\kern-0.25ex\vert #1 
    \vert\kern-0.25ex\vert\kern-0.25ex\vert}}
    
\newtheorem{prpstn}{Proposition}[section]
\newtheorem{lmm}{Lemma}[section]
\newtheorem{thrm}{Theorem}[section]
\newtheorem{dfntn}{Definition}[section]
\newtheorem{crllr}{Corollary}[section]
\newtheorem{rmrk}{Remark}[section]
\everymath{\displaystyle}

\begin{document}
\title{Internal rapid stabilization of a 1-D linear transport equation with a scalar feedback}

\author{{C}hristophe {Z}hang}

\maketitle
\begin{abstract} We use the backstepping method to study the stabilization of a 1-D linear transport equation on the interval $(0,L)$, by controlling the scalar amplitude of a piecewise regular function of the space variable in the source term. We prove that if the system is controllable in a periodic Sobolev space of order greater than $1$, then the system can be stabilized exponentially in that space and, for any given decay rate, we give an explicit feedback law that achieves that decay rate.  \end{abstract}

\section{Introduction}
We study the linear 1-D hyperbolic equation
\begin{equation}\label{GeneralSystem}
\left\{\begin{aligned}
y_t + y_x + a(x) y& = u(t) \tilde{\varphi} (x), \ x\in [0,L], \\
y(t,0) & = y(t,L), \ \forall t \geq 0,
\end{aligned}\right.
\end{equation}
where  $a$ is continuous, real-valued, $\tilde{\varphi}$ is a given real-valued function that will have to satisfy certain conditions, and at time $t$, $y(t, \cdot)$ is the state and $u(t)$ is the control. As the system can be transformed into
\begin{equation}\label{System}
\left\{\begin{aligned}
\alpha_t + \alpha_x + \mu \alpha & = u(t) \varphi (x), \ x\in [0,L], \\
\alpha(t,0) & = \alpha(t,L), \ \forall t \geq 0,
\end{aligned}\right.
\end{equation}
through the state transformation
$$\alpha(t,x) := e^{\int_0^x a (s)ds - \mu x} y(x,t),$$
where $\mu=\int_0^L a(s)ds$, and with
$$\varphi(x):= e^{\int_0^x a (s)ds- \mu x} \tilde{\varphi} (x),$$
we will focus on systems of the form \eqref{System} in this article.

{\color{black}These systems are an example of linear hyperbolic systems with a distributed scalar input. Such systems appear naturally in physical problems. For example, as is mentioned in \cite{Russell1}, a linear wave equation which can be rewritten as a $2\times 2$ first order hyperbolic system, the problem of a vibrating damped string, or the plucking of a string, can be modelled thus. In a different field altogether, chemical tubular reactors, in particular plug flow reactors (see \cite{OrlovDochain, Pitarch}), are modeled by hyperbolic systems with a distributed scalar input (the temperature of the reactor jacket), albeit with a boundary input instead of proportional boundary conditions. Let us cite also the water tank system, introduced by François Dubois, Nicolas Petit and Pierre Rouchon in \cite{DPR1999}. It models a 1-D tank containing an inviscid, incompressible, irrotational fluid, in the approximation that its acceleration is small compared with the gravitational constant, and that the height of the liquid is small compared with the length of the tank. In this setting, the motion of the fluid can be modelled by the Saint-Venant equations on the interval $[0,L]$ with impermeable boundary conditions (which correspond to proportional boundary conditions after a variable change), and the control is the force applied to the tank itself, which takes the form of a distributed scalar input. }

\subsection{Notations and definitions}
We note $\ell^2$ the space of summable square series $\ell^2 (\Z)$.
To simplify the notations, we will note $L^2$ the space $L^2(0,L)$ of complex-valued $L^2$ functions, with its hermitian product
\begin{equation}\label{ScalarProduct}
 \langle f, g \rangle = \int_0^L f(x)\overline{g(x)}dx, \quad \forall f, g\in L^2,
\end{equation}
and the associated norm $\|\cdot\|$.
We also use the following notation
\begin{equation}
e_n(x)=\frac1{\sqrt{L}} e^{\frac{2i\pi}{L} n x}, \quad \forall n \in \Z, 
\end{equation}
the usual Hilbert basis for $L^2$. For a function $f\in L^2$, we will note $(f_n) \in \ell^2$ its coefficients in this basis: 
$$f=\sum_{n \in \Z} f_n e_n.$$

Note that with this notation, we have
$$\bar{f}=\sum_{n \in \Z} \overline{f_{-n}} e_n,$$
so that, in particular, if $f$ is real-valued:
$$ f_{-n} = \overline{f_n}, \quad \forall n \in \Z.$$

Functions of $L^2$ can also be seen as $L$-periodic functions on $\RR$, by the usual $L$-periodic continuation: in this article, for any $f\in L^2$ we will also note $f$ its $L$-periodic continuation on $\RR$.

We will use the following definition of the convolution product on $L$-periodic functions:
\begin{equation}\label{Conv}
f \star g = \sum_{n\in \Z} f_n g_n e_n= \int_{0}^L f(s)g(\cdot -s) ds \in L^2,\quad \forall f, g\in L^2,  
\end{equation}
where $g(x-s)$ should be understood as the value taken in $x-s$ by the $L$-periodic continuation of $g$.

Let us now note $\mathcal{E}$ the space of finite linear combinations of the $(e_n)_{n \in \Z}$. Then, any sequence $(f_n)_{n \in \Z}$ defines an element $f$ of $\mathcal{E}^\prime$:
$$\langle e_n, f \rangle =\overline{f_n}.$$ 
On this space of linear forms, derivation can be defined by duality:
$$f^\prime= \left(\frac{2i\pi n}L f_n \right)_{n\in \Z}, \quad \forall f \in \dist.$$
We also define the following spaces:
\begin{dfntn}
Let $m \in \NN$. 
We note $H^m$ the usual Sobolev spaces on the interval $(0,L)$, equipped with the Hermitian product
$$ \langle f, g\rangle_m = \int_0^L \partial^m f \overline{\partial^m g}+ f \overline{g}, \quad \forall f,g \in H^m, $$
and the associated norm $\|\cdot\|_m$.

For $m\geq1$ we also define $H^m_{(pw)}$ the space of piecewise $H^m$ functions, that is, $f\in H^m_{(pw)}$ if there exists a finite number $d$ of points $(\sigma_j)_{1 \leq j \leq d} \in [0, L]$ such that, noting $\sigma_0:=0$ and $\sigma_{d+1}:=L$, $f$ is $H^m$ on every $[\sigma_j, \sigma_{j+1}]$ for $0\leq j \leq d$.
This space can be equipped with the norm 
$$\|f\|_{m, pw}:=\sum_{j=0}^d \|f_{|[\sigma_j, \sigma_{j+1}]} \|_{H^m(\sigma_j, \sigma_{j+1})}.$$

For $s>0$, we also define the periodic Sobolev space $H^s_{per}$ as the subspace of $L^2$ functions $f=\sum_{n \in \Z} f_n e_n$ such that
$$\sum_{n \in \Z} \left(1+\left|\frac{2i\pi n}{L}\right|^{2s}\right)|f_n|^2 < \infty. $$
$H^s$ is a Hilbert space, equipped with the Hermitian product
$$ \langle f, g\rangle_s = \sum_{n \in \Z} \left(1+\left|\frac{2i\pi n}{L}\right|^{2s}\right) f_n \overline{g_n}, \quad \forall f,g \in H^s,$$
and the associated norm $\|\cdot\|_s$, as well as the Hilbert basis $$(e_n^s):=\left(\frac{e_n}{\sqrt{1+\left|\frac{2i\pi n}{L}\right|^{2s}}}\right).$$
\end{dfntn}

Note that for $m \in \NN$, $H^m_{per}$ is a closed subspace of $H^m$, with the same scalar product and norm, thanks to the Parseval identity. Moreover,
$$H^m_{per}=\left\{f\in H^m, \quad {\color{black} f^{(i)}(0)=f^{(i)}(L), \ \forall i \in \{0, \cdots, m-1\}}\right\}.$$

\subsection{Main result}

To stabilize \eqref{System}, we will be considering linear feedbacks of the form
$$\langle  \alpha(t), F \rangle = \sum_{n\in \Z} \overline{F_n} \alpha_n(t)=\int_{0}^L \bar{F}(s) \alpha(s)ds$$
where $F\in \dist$ and $(F_n) \in \CC^\Z$ are its Fourier coefficients, and $F$ is real-valued, that is,
$$ F_{-n} = \overline{F_n}, \quad \forall n \in \Z.$$
In fact, the integral notation will appear as purely formal, as the $(F_n)$ will have a prescribed growth, so that $F \notin L^2$. The associated closed-loop system now writes
\begin{equation}\label{ClosedLoop}
\left\{\begin{aligned}
\alpha_t + \alpha_x +\mu \alpha & = \langle \alpha (t), F \rangle \varphi (x), \ x\in [0,L], \\
\alpha(t,0) & = \alpha(t,L), \ \forall t \geq 0.
\end{aligned}\right.
\end{equation}
This is a linear transport equation, which we seek to stabilize with an internal, scalar feedback, given by a real-valued feedback law. This article aims at proving the following class of stabilization results:
\begin{thrm}[Rapid stabilization in Sobolev norms]\label{MainResult}
Let $m \geq 1$. Let $\varphi \in H^m_{(pw)} \cap H^{m-1}_{per}$ such that
\begin{equation}\label{phiGrowth}
 \frac{c}{\sqrt{1+\left|\frac{2i \pi n}{L}\right|^{2m}}} \leq |\varphi_n| \leq \frac{C}{\sqrt{1+\left|\frac{2i \pi n}{L}\right|^{2m}}}, \quad \forall n \in \Z,
\end{equation}
where $c,C >0$ are the optimal constants for these inequalities.
Then, for every $\lambda {\color{black}>} 0$, {\color{black}for all $\alpha_0 \in H^{m}_{per}$ the closed-loop system \eqref{ClosedLoop} with the stationary feedback law $F\in \dist$ given by
\begin{equation*}
\langle e_n, F \rangle:=-\frac{1-e^{\-\lambda L}}{1+e^{-\lambda L}}\frac{2}{L\overline{\varphi_n}}, \quad \forall n \in \Z,
\end{equation*}
has a solution $\alpha(t)$ which satisfies the estimate
\begin{equation}\label{exp-stab-estimate} 
\| \alpha(t) \|_{m} \leq \left(\frac{C}{c}\right)^2e^{(\mu+\lambda)  L}  e^{-\lambda t} \|\alpha_0\|_{m},\quad \forall t \geq 0.\end{equation}}
\end{thrm}
Note that the estimate \eqref{exp-stab-estimate} is constructive, as it only depends on $c, C, \mu$ and $\lambda$. Though it is not necessarily sharp for a given controller $\varphi$ and the corresponding feedback law $F$, it is the ``least worse'' a priori estimate one can get, in a sense that we will elaborate further on.
The growth restriction \eqref{phiGrowth} on the Fourier coefficients of $\varphi$ can be written, more intuitively, and for some other constants $c^\prime, C^\prime >0$, 
 $$ \frac{c^\prime}{1+\left|\frac{2i \pi n}{L}\right|^{m}} \leq |\varphi_n| \leq \frac{C^\prime}{1+\left|\frac{2i \pi n}{L}\right|^{m}}, \quad \forall n \in \Z,$$
 and corresponds to the necessary and sufficient condition for the controllability of system \eqref{System} in $H^{m}_{per}$, in time $T\geq L$, {\color{black}with $L^2(0,T)$ controls. This is obtained using the moments method, and we refer to \cite[Equation (2.19) and pages 199-200]{Russell2} for more details. The controllability of system \eqref{System}, in turn, will allow us to use a form of backstepping method to stabilize it. }

On the other hand, the additional regularity $\varphi \in H^m_{(pw)}$ gives us the following equality, {\color{black}first using the fact that $\varphi \in H^{m-1}_{per}$, then by integration by parts on each interval $[\sigma_j, \sigma_{j+1}]$, using the fact that $\partial^{m-1}\varphi \in H^1_{(pw)}$:
\begin{equation}\label{DAphi}
\begin{aligned}
   \varphi_n&=\frac{(-1)^{m-1}}{\left(\frac{2i\pi n}L\right)^{m-1}}\langle \varphi, \partial^{m-1}e_n\rangle\\
   &=\frac{1}{\left(\frac{2i\pi n}L\right)^{m-1}}\langle \partial^{m-1}\varphi, e_n\rangle\\
   &=- \frac{\tau^\varphi_n}{\left(\frac{2i\pi}{L}n\right)^m} +  \frac{1}{\left(\frac{2i\pi}{L}n\right)^m} \sum_{j=0}^{d}\left\langle \chi_{[\sigma_j, \sigma_{j+1}]}\partial^m \varphi, e_n\right\rangle, \quad \forall n\in \Z^\ast,
   \end{aligned}
\end{equation}
where 
$$\tau^\varphi_n:= \frac1{\sqrt{L}}\left( \partial^{m-1}\varphi(L)-\partial^{m-1}\varphi(0) + \sum_{j=1}^d e^{-\frac{2i\pi}L n\sigma_j} (\partial^{m-1} \varphi (\sigma_j^-) - \partial^{m-1}\varphi(\sigma_j^+))\right), \quad \forall n \in \Z.$$}
Note that, thanks to condition \eqref{phiGrowth}, there exists $C_1, C_2 >0$ such that 
$$C_1 \leq |\tau^\varphi_n| \leq C_2, \quad n\in \Z,$$
so that these numbers are the eigenvalues of a diagonal isomorphism of any {\color{black}periodic} Sobolev space into itself, which we note $\tau^\varphi$. {\color{black}Moreover, it is clear from the definition of its coefficients that $\tau^\varphi$ is a sum of translations.}
Also, note that $\tau^\varphi_n \neq 0$, and thus, $\varphi \notin H^m_{per}$. Finally, note that
\begin{equation}\label{ResteDA}
\left(\sum_{j=0}^{d}\left\langle \chi_{[\sigma_j, \sigma_{j+1}]}\partial^m \varphi, e_n\right\rangle\right) \in \ell^2.\end{equation}

\subsection{Related results}

{\color{black}To investigate the stabilization of infinite-dimensional systems, there are four main types of approaches.

The first type of approach relies on abstract methods, such as the Gramian approach and the Riccati equations (see for example \cite{Vest, Urquiza,Komornik}). In these works, rapid stabilization was achieved thanks to a generalization of the well-known Gramian method in finite dimension (see \cite{Lukes, Kleinman}). However, the feedback laws that are provided involve the solution to an algebraic Riccati equation, and the inversion of an infinite-dimensional Gramian operator, which makes them difficult to compute in practice.

The second approach relies on Lyapunov functions. Many results on the boundary stabilization of first-order hyperbolic systems, linear and nonlinear, have been obtained using this approach: see for example the book \cite{BastinCoronBook}, and the recent results in \cite{Hayat1,Hayat2}. However, this approach can be limited, as it is sometimes impossible to obtain an arbitrary decay rate using Lyapunov functions (see \cite[Remark 12.9, page 318]{CoronBook} for a finite dimensional example). 

The third approach is related to pole-shifting results in finite dimension. Indeed, it is well-known that if a linear finite-dimensional system is controllable, than its poles can be arbitrarily reassigned (shifted) with an appropriate linear feedback law (see \cite{CoronBook}). There have been some generalizations of this powerful property to infinite-dimensional systems, notably hyperbolic systems. Let us cite \cite{Russell2}, in which the author uses a sort of canonical form to prove a pole-shifting result for a class of hyperbolic systems with a distributed scalar control. In this paper, the feedback laws under consideration are bounded and pole-shifting property is not as strong as in finite dimension. This is actually inevitable, as was proved in \cite{Sun}, in a very general setting: bounded feedback laws can only achieve weak pole-shifting, which is not sufficient for exponential stabilization. However, if one allows for unbounded feedback laws, it is possible to obtain stronger pole-shifting, and in particular exponential stabilization in some cases. This is extensively studied in \cite{Rebarber}, in which the author gives a formula for a feedback law that achieves the desired pole placement. However, this formula requires to know a cardinal function for which the poles coincide with the initial spectrum, which might be difficult in practice. 

The fourth approach, which we will be using in this article, is the backstepping method. This name originally refers to a way of designing feedbacks for finite-dimensional stabilizable systems with an added chain of integrators (see \cite{CoronBook, Sontag, KrsticKanellakopoulosKokotovic}, and \cite{CoronAN} or \cite{KrsticLiu2} for some applications to partial differential equations). Another way of applying this approach to partial differential equations was then developed in \cite{KBHeat} and \cite{KBParab}: when applied to the discretization of the heat equation, the backstepping approach yielded a change of coordinates which was equivalent to a Volterra transform  of the second kind. Backstepping then took yet another successful form, consisting in mapping the system to stable target system, using a Volterra transformation of the second kind (see \cite{KrsticBook} for a comprehensive introduction to the method):
$$f(t,x) \mapsto f(t,x)-\int_0^x k(x,y) f(t,y) dy.$$
This was used to prove a host of results on the boundary stabilization of partial differential equations: let us cite for example \cite{KrsticWave} and \cite{KrsticWave2} for the wave equation, \cite{XiangKdVNull,XiangKdVFinite} for the Korteweg-de Vries equation, \cite[chapter 7]{BastinCoronBook} for an application to first-order hyperbolic systems, and also \cite{BCKV}, which combines the backstepping method with Lyapunov functions to prove finite-time stabilization in $H^2$ for a quasilinear $2\times 2$ hyperbolic system. 

The backstepping method has the advantage of providing explicit feedback laws, which makes it a powerful tool to prove other related results, such as null-controllability or small-time stabilization (stabilization in an arbitrarily small time). This is done in \cite{CoronNguyen}, where the authors give an explicit control to bring a heat equation to $0$, then a time-varying, periodic feedback to stabilize the equation in small time. In \cite{XiangKdVFinite}, the author obtains the same kind of results for the Korteweg-de Vries equation.

In some cases, the method was used to obtain stabilization with an internal feedback. This was done in \cite{KrsticHaraTsub} and \cite{Woittennek} for parabolic systems, and \cite{YuXuJiangGanesan} for first-order hyperbolic systems. The strategy in these works is to first apply a Volterra transformation as usual, which still leaves an unstable source term in the target, and then apply a second invertible transformation to reach a stable target system. Let us note that in the latter reference, the authors study a linear transport equation and get finite-time stabilization. However, their controller takes a different form than ours, and several hypotheses are made on the space component of the controller so that a Volterra transform can be successfully applied to the system. This is in contrast with the method in this article, where the assumption we make on the controller corresponds to the exact null-controllability of the system. 

In this paper, we use another application of the backstepping method, which uses another type of linear transformations, namely, Fredholm transformations:
$$f(t,x) \mapsto \int_0^L k(x,y) f(t,y) dy.$$
These are more general than Volterra transformations, but they require more work: indeed, Volterra transformations are always invertible, but the invertibility of a Fredholm transform is harder to check.  Even though it is sometimes more involved and technical, the use of a Fredholm transformation proves more effective for certain types of control: for example, in \cite{CoronLu1} for the Korteweg-de Vries equation and \cite{CoronLu2} for a Kuramoto-Sivashinsky, the position of the control makes it more appropriate to use a Fredholm transformation. Other boundary stabilization results using a Fredholm transformation can be found in \cite{CHO1} for integro-differential hyperbolic systems, and in \cite{CHO2} for general hyperbolic balance laws. 

Fredholm transformations have also been used in \cite{Schrodinger}, where the authors prove the rapid stabilization of the Schrödinger equation with an internal feedback. Their method of proof relies on the assumption that the system is controllable, and the technical developments are quite different from the work in previous references. 
This is a new development in the evolution of the backstepping method. Indeed, the original form of the backstepping method, and the backstepping method with Volterra transformations of the second kind, could be applied to uncontrollable systems. Hence, a controllability assumption makes for potentially powerful additional information, for example when one considers the more general Fredholm transformations instead of Volterra transformations of the second kind. It is interesting to note that the role played by controllability is also a feature of the pole-shifting approach and the Gramian method, although in this setting it leads to an explicit feedback law given by its Fourier coefficients, instead of the inverse of the Gramian operator, or, in the case of Richard Rebarber's result in \cite{Rebarber}, a cardinal function.

\subsection{The backstepping method revisited: a finite-dimensional example}\label{FiniteD}
Let us now give a finite-dimensional example to illustrate the role controllability can play in the backstepping method for PDEs.
Consider the finite-dimensional control system
\begin{equation}\label{FiniteDControlSys}
\dot{x}=Ax+ Bu(t), \quad x\in \CC^n, A \in M_n (\CC), B\in M_{n,1} (\CC).
\end{equation}
Assume that $(A,B)$ is controllable. Suppose that $x(t)$ is a solution of system \eqref{FiniteDControlSys} with $u(t)=Kx(t)$. Now, in the spirit of PDE backstepping, let us try to invertibly transform the resulting closed-loop system into another controllable system, namely
\begin{equation}\label{FiniteDTarget}
\dot{x}=\tilde{A}x,
\end{equation}
which can be exponentially stable if $\tilde{A}$ is well chosen. 

 Such a transformation $T$ would map the closed loop system to
$$\dot{(Tx)} = T\dot{x}=T(A+BK)x .$$
In order for $Tx$ to be a solution of \eqref{FiniteDTarget}, we would need
\begin{equation}\label{FirstFiniteDOpEq}
   T(A+BK)=\tilde{A}T.
\end{equation}
One can see quite clearly that this matrix equation is not well-posed, in that if it has a solution, it has an infinity of solutions. Moreover, the variables $T$ and $K$ are not separated because of the $TBK$ term, and as a result the equation is nonlinear. Hence, we can add the following constraint to equation \eqref{FirstFiniteDOpEq}, to separate the variables, make the equation linear in $(T,K)$, and get a uniqueness property:
\begin{equation}
    TB=B.
\end{equation}
Injecting the above equation into \eqref{FirstFiniteDOpEq}, we get the following equations:
\begin{equation}\label{FiniteDOpEq}
\begin{aligned}
TA+BK&=\tilde{A}T, \\
TB&=B,
\end{aligned}
\end{equation}
Now for this set of equations, one can prove the following theorem, using the Brunovski normal form (or canonical form):
\begin{thrm}
If $(A,B)$ and $(\tilde{A}, B)$ are controllable, then there exists a unique pair $(T,K)$ satisfying conditions \eqref{FiniteDOpEq}.
\end{thrm}

This shows that controllability can be very useful when one wants to transform systems into other systems. In the finite-dimensional case, using the canonical form is the most efficient way of writing it. However, in order to gain some insight on the infinite-dimensional case, there is a different proof, relying on the spectral properties of $A$ and $\tilde{A}$, which can be found in \cite{Schrodinger}. The idea is that the controllability of $A$ allows to build a basis for the space state, in which $T$ can then be constructed.
Indeed, suppose $A$ is diagonalizable with eigenvectors $(e_n, \lambda_n)_{1\leq n \leq N}$, and suppose that $\tilde{A}$ and $A$ have no mutual eigenvalues. Then, let us project \eqref{FiniteDOpEq} on $e_n$:
\begin{equation}
    \label{Proj-Kern-FiniteD}
\lambda_n Te_n + (Ke_n)B=\tilde{A} Te_n , 
\end{equation}
from which we get the following relationship
\begin{equation}
\label{Basis-FiniteD}
    Te_n=(Ke_n)(\tilde{A}-\lambda_n I)^{-1} B, \quad \forall n \in \{1, \cdots, N \}.
\end{equation}
Then, using the Kalman rank condition on the pair $(\tilde{A},B)$, one can prove that the $f_n:=((\tilde{A}-\lambda_n I)^{-1} B)$ form a basis of $\RR^N$. 

Knowing this, write
\begin{equation}
\begin{aligned}
    B&=\sum_{n=1}^N b_n e_n, \\
    B&=\sum_{n=1}^N \tilde{b}_n f_n,
\end{aligned}
\end{equation}
and $TB$ is written naturally in this basis:
\begin{equation}
    TB=\sum_{n=1}^N (Ke_n)b_nf_n,
\end{equation}
so that
the second equation of \eqref{FiniteDOpEq} becomes
\begin{equation}
    \sum_{n=1}^N (Ke_n)b_nf_n=\sum_{n=1}^N \tilde{b}_n f_n.
\end{equation}
Using the Kalman rank condition on $(A,B)$, one can prove that $b_n \neq 0$ so that the $(Ke_n)$ are uniquely determined. The only thing that remains to prove is the invertibility of $T$, as the $(Ke_n)$ could be $0$. In the end the invertibility is proven thanks to the Hautus test on the pair $(\tilde{A}, B)$, and the uniqueness is given by the $TB=B$ condition. 

\begin{rmrk}
In this finite-dimensional example, one can see a relationship between the Gramian method and our variant of the backstepping method. Indeed, the Gramian matrix, defined by 
\begin{equation*}
    C_\omega^\infty:=\int_0^\infty e^{-2\omega t} e^{-tA} B B^\ast e^{-tA^\ast} dt, 
\end{equation*}
is the solution of the Lyapunov equation:
\begin{equation*}
   C_\omega^\infty(A+\omega I)^\ast + (A+\omega I) C_\omega^\infty=B B^\ast.
\end{equation*}

Now, injecting the feedback law $K:=-B^\ast ( C_\omega^\infty)^{-1}$ given by the Gramian method, this equation becomes:
\begin{equation*}
   C_\omega^\infty(A+\omega I)^\ast + (A+\omega I) C_\omega^\infty=-B K C_\omega^\infty,
\end{equation*}
which becomes, after multiplication by $(C_\omega^\infty)^{-1}$ on the left and on the right,
\begin{equation*}
    (C_\omega^\infty)^{-1}(A+BK)=(-A^\ast- 2\omega I)(C_\omega^\infty)^{-1},
\end{equation*}
which is of the form \eqref{FirstFiniteDOpEq}, with $\tilde{A}=-A^\ast-2\omega I$ and $T=(C_\omega^\infty)^{-1}$. The fundamental difference then comes from the fact that in the Gramian method, this backstepping-type equation is coupled with the definition of the feedback law
\begin{equation*}
    B^\ast (C_\omega^\infty)^{-1}=-K,
\end{equation*}
whereas in the backstepping method, the backstepping-type equation \eqref{FirstFiniteDOpEq} is coupled with the $TB=B$ condition, which can be recast as
\begin{equation*}
    B^\ast T^\ast = B^\ast.
\end{equation*}
The former leads to a Lyapunov-type analysis of stability, whereas the latter allows us to use Fourier analysis to give explicit coefficients for the feedback law. 
\end{rmrk}

\subsection{Structure of the article}

The structure of this article follows the outline of the proof given above: in Section 2, we look for candidates for the backstepping transformation in the form of Fredholm transformations. Formal calculations (and a formal $TB=B$ condition) lead to a PDE analogous to \eqref{Proj-Kern-FiniteD} which we solve, which is analogous to the derivation of \eqref{Basis-FiniteD}. 
Using the properties of Riesz bases and the controllability assumption, we prove that such candidates are indeed invertible, under some conditions on the feedback coefficients $(F_n)$. For consistency, we then determine the feedback law $(F_n)$ such that the corresponding transformation indeed satisfies a weak form of the $TB=B$ condition.
Then, in Section 3, we check that the corresponding transformation indeed satisfies an operator equality  analogous to \eqref{FiniteDOpEq}, making it a valid backstepping transformation. 
We check the well-posedness of the closed-loop system for the feedback law obtained in Section 2, which allows us to prove the stability result. Finally, Section 4 gives a few remarks on the result, as well as further questions on this stabilization problem.
}
\section{Definition and properties of the transformation}
Let $\lambda^\prime >0$ {\color{black}be such that $\lambda^\prime - \mu >0$}, and $m \geq 1$.  Let $\varphi \in H^m \cap H^{m-1}_{per}$ be a real-valued function satisfying \eqref{phiGrowth}.
We consider the following target system:
\begin{equation}\label{Target1}
\left\{\begin{aligned}
z_t+z_x +\lambda^\prime z& =0, \quad x\in (0,L), \\
z(t,0)& =z(t,L), \quad t\geq 0.\end{aligned}\right.
\end{equation}
Then it is well-known that, taking $\alpha_0 \in L^2$, the solution to \eqref{Target1} with initial condition $\alpha_0$ writes
$$ z(t,x)= e^{-\lambda^\prime t} \alpha_0 (x -t), \quad \forall (t,x) \in \RR^{+} \times (0,L).$$
Hence,
\begin{prpstn}For all $s \geq 0$, the system \eqref{Target1} is exponentially stable for $\|\cdot \|_s$, for initial conditions in $H^s_{per}$.\end{prpstn}

\subsection{Kernel equations}\label{FindKernel}
As mentioned in the introduction, we want to build backstepping transformations $T$ as a kernel operator of the Fredholm type:
$$f(t,x) \mapsto \int_0^L k(x,y) f(t,y) dy.$$
To have an idea of what this kernel looks like, we can do the following formal computation for some Fredholm operator $T$: first the boundary conditions
$$\left(\int_{0}^L k(0,y)\alpha(y) dy\right)=\left(\int_{0}^L k(L,y)\alpha(y) dy\right),$$
then the equation of the target system, for $x\in [0,L]$:
{\def\arraystretch{3}$$\begin{array}{rcl}
 0&=& \left(\int_{0}^L k(x,y)\alpha(y) dy\right)_t + \left(\int_{0}^L k(x,y)\alpha(y) dy\right)_x +\lambda^\prime \left(\int_{0}^L k(x,y)\alpha(y) dy\right)   \\
&=& \left(\int_{0}^L k(x,y)\alpha_t(y) dy\right) + \left(\int_{0}^L k_x(x,y)\alpha(y) dy\right)+\lambda^\prime \left(\int_{0}^L k(x,y)\alpha(y) dy\right) \\
&=& \left(\int_{0}^L k(x,y)(-\alpha_x(y) - \mu \alpha(y) + \langle \alpha, F \rangle \varphi(y)) dy\right)+ \left(\int_{0}^L (k_x(x,y)+\lambda^\prime k(x,y))\alpha(y) dy\right) 
\end{array}$$}
{\def\arraystretch{3}$$\begin{array}{rcl}
&=& \left(\int_{0}^L k_y(x,y)\alpha(y)  dy\right) - (k(x,L) \alpha (L)-k(x,0)\alpha(0))+ \left(\int_{0}^L k(x,y)\langle \alpha, F \rangle \varphi(y)) dy\right)+\\
& & \left(\int_{0}^L (k_x(x,y)+(\lambda^\prime-\mu) k(x,y))\alpha(y) dy\right)\\
&=&  \left(\int_{0}^L k(x,y)\left(\int_0^L \bar{F}(s)\alpha(s) ds\right) \varphi(y)) dy\right)- (k(x,L) \alpha (L)-k(x,0)\alpha(0))\\
& & + \left(\int_{0}^L (k_y(x,y)+k_x(x,y)+(\lambda^\prime-\mu) k(x,y))\alpha(y) dy\right)\\
&=&  \left(\int_{0}^L \bar{F}(s)\left(\int_0^L k(x,y) \varphi(y) dy\right) \alpha(s) ds\right)- (k(x,L) \alpha (L)-k(x,0)\alpha(0))\\
& & + \left(\int_{0}^L (k_y(x,y)+k_x(x,y)+(\lambda^\prime-\mu) k(x,y))\alpha(y) dy\right).
\end{array}$$}
Now, suppose we have the formal $TB=B$ condition
$$ \int_0^L k(x,y) \varphi(y) dy = \varphi(x), \quad \forall x \in [0,L].$$
Then, we get, noting $\lambda:=\lambda^\prime-\mu{\color{black} >0}$,
$$\left(\int_{0}^L \left(k_y(x,y)+k_x(x,y)+\lambda k(x,y) + \varphi(x)\bar{F}(y)\right)\alpha(y) dy\right)- (k(x,L) \alpha (L)-k(x,0)\alpha(0)) = 0. $$
Hence the kernel equation:
\begin{equation}\label{Kernel1}
\left\{\begin{aligned}
k_x+k_y+\lambda k&=- \varphi(x)\bar{F}(y), \\
k(0,y) & =k(L,y), \\
k(x,0)& =k(x,L),\end{aligned}\right.
\end{equation}
together with the $TB=B$ condition
\begin{equation}\label{TB}
\left\langle k(x, \cdot), \varphi(\cdot) \right\rangle = \varphi(x), \quad \forall x \in [0,L].
\end{equation}

\subsection{Construction of Riesz bases for Sobolev spaces}

To study the solution to the kernel equation, we project it along the variable $y$. Let us write heuristically 
$$k(x,y)=\sum_{n \in \Z} k_{n}(x)e_n(y),$$
so that
$$\int_0^L k(x,y)\alpha(y)dy = \sum_{n\in \Z} \alpha_n k_{-n}(x).$$
Projecting the kernel equations \eqref{Kernel1}, we get
\begin{equation}\label{kn1}
k_{n}^\prime + \lambda_n k_{n} = -\overline{F_{-n}} \varphi,
\end{equation}
where
\begin{equation}\label{lambda_n}
\lambda_n=\lambda + \frac{2i\pi}{L} n.
\end{equation}
Note that 
\begin{equation}\label{ElementaryCoeff}
  \frac{2i\pi p}{L} \frac1{\lambda_{n+p}} + \lambda_n \frac1{\lambda_{n+p}} = 1,\quad \forall n, p  \in \Z.
\end{equation}
Now consider the $L^2$ function given by
\begin{equation}\label{Elementary}
\Lambda^{\lambda}_n(x)=\frac{\sqrt{L}}{1-e^{-\lambda L}}e^{-\lambda_n x},  \quad \forall n \in \Z, \quad\forall x \in [0, L).
\end{equation}
Then, for all $m\geq 0$, $\Lambda^{\lambda}_n \in H^m$, and we have
$$ \langle \Lambda^{\lambda}_n  , e_p \rangle = \frac1{\sqrt{L}}\int_0^L\frac{\sqrt{L}}{1-e^{-\lambda L}}e^{-\lambda_n x} e^{-\frac{2i\pi p}{L}x} dx = \frac{1}{1-e^{-\lambda L}}\int_0^L e^{-\lambda_{n+p} x}dx= \frac1{\lambda_{n+p}}, \quad \forall n, p  \in \Z,  $$
so that, using \eqref{ElementaryCoeff},
$$(\Lambda^{\lambda}_n)^\prime + \lambda_n \Lambda^{\lambda}_n = \sum_{p \in \Z} e_p \ \textrm{in} \  \dist.$$

\begin{rmrk}
In $\dist$, $\sum_{p \in \Z} e_p$ is the equivalent of the Dirac comb, or the ``Dirac distribution'' on the space of functions on $[0,L]$. So, in a sense, $\Lambda^{\lambda}_n$ is the elementary solution of \eqref{kn1}.
\end{rmrk}

Let us now define, in analogy with the elementary solution method,
\begin{equation}\label{kndef}
 k_{n, \lambda}=-\overline{F_{-n}} \Lambda^{\lambda}_n \star \varphi \in H^m_{per}, \quad \forall n \in \Z.
\end{equation}
The regularity comes from the definition of the convolution product, \eqref{phiGrowth} and \eqref{lambda_n}, and one can check, using \eqref{ElementaryCoeff}, that $k_{n, \lambda}$ is a solution of \eqref{kn1}.

The next step to build an invertible transformation is to find conditions under which $(k_{n, \lambda})$ is some sort of basis. More precisely we use the notion of Riesz basis (see \cite[Chapter 4]{Christensen})
\begin{dfntn} 
A Riesz basis in a Hilbert space $H$ is the image of an orthonormal basis of $H$ by an isomorphism.
\end{dfntn}
\begin{prpstn}\label{Riesz}
Let $H$ be a Hilbert space. A family of vectors $(f_k)_{k \in \NN} \in H$ is a Riesz basis if and only if it is complete (i.e., $ \overline{\textup{Span}(f_k)}=H $) and there exists constants $C_1, C_2 >0$ such that, for any scalar sequence $(a_k)$ with finite support, 
\begin{equation}\label{RieszBounds}
C_1 \sum |a_k|^2 \leq \left\|\sum a_k f_k \right\|_H^2 \leq C_2 \sum |a_k|^2.
\end{equation}
\end{prpstn}
Let us now introduce the following growth condition:
\begin{dfntn}
Let $s \geq 0$, $(u_n)\in \CC^\Z$ (or $u\in \dist$) . We say that $(u_n)$ (or $u$) has $s$-growth if 
\begin{equation}\label{FGrowth}
 c\sqrt{1+\left|\frac{2i \pi n}{L}\right|^{2s}}\leq |u_n| \leq C\sqrt{1+\left|\frac{2i \pi n}{L}\right|^{2s}}, \quad \forall n \in \Z, 
\end{equation}
for some $c,C>0$. The optimal constants for these inequalities are called growth constants.
\end{dfntn}

\begin{rmrk}
The inequalities \eqref{FGrowth} can also be written, more intuitively, and for some other positive constants,
\begin{equation}
 c\left(1+|n|^{s}\right)\leq |u_n| \leq C\left(1+|n|^{s}\right), \quad \forall n \in \Z.
\end{equation}
\end{rmrk}
 
 We can now establish the following Riesz basis properties for the $(k_{n, \lambda})$:
\begin{prpstn}\label{RieszFam}
Let $s\geq 0$.
If $(F_n)$ has $s$-growth, then the family of functions
$$(k_{n, \lambda}^s):= \left(\frac{k_{n, \lambda}}{\sqrt{1+\left|\frac{2i\pi n}{L}\right|^{2s}}}\right)$$
is a Riesz basis for $H^{m}_{per}$.
\end{prpstn}
\begin{proof}
We use the characterization of Riesz bases given in Proposition \ref{Riesz}. 
First, let us prove the completeness of $(k_{n, \lambda}^s)$. Let $f\in H^m_{per}$ be such that
$$\langle f, k_{n, \lambda}^s \rangle_{m} =0, \quad \forall n \in \Z.$$
Then for all $n\in \Z$ we get
$$0=\langle \Lambda_{n}^{\lambda} \star \varphi, f \rangle_{m}=\sum_{p\in \Z} \frac{\overline{f_p} \varphi_p}{\lambda_{n+p}} \left(1+\left|\frac{2i \pi p}{L}\right|^{2m}\right)=\left\langle \Lambda_{n}^{\lambda}, \sum_{p\in Z} \left(1+\left|\frac{2i \pi p}{L}\right|^{2m}\right)f_p\overline{\varphi_p} e_p \right\rangle,$$
as, thanks to \eqref{phiGrowth}, and using the fact that $f\in H^m_{per}$, 
$$ \sum_{p\in Z} \left(1+\left|\frac{2i \pi p}{L}\right|^{2m}\right)f_p\overline{\varphi_p} e_p \in L^2.$$
Now, $(\Lambda_{n}^{\lambda})$ is a complete family of $L^2$, as it is a Riesz basis, so that
$$f_p \varphi_p = 0, \quad \forall p \in \Z.$$
Recalling condition \eqref{phiGrowth}, this yields
$$ f_p  = 0, \quad \forall p \in \Z,$$
which proves the completeness of $(k_{n, \lambda}^s)$.

Now let $I \subset \Z$ be a finite set, and $(a_n) \in \CC^I$. Then,
{\def\arraystretch{3}$$\begin{array}{rcl}
\left\|\sum_{n\in I} a_n k_{n, \lambda}^s\right\|_m^2 & =& \left\| \sum_{n\in I} -a_n \frac{\overline{F_{-n}}}{\sqrt{1+\left|\frac{2i\pi n}{L}\right|^{2s}}} \Lambda^{\lambda}_n \star \varphi \right\|_{m}^2 \\
&=& \left\| \sum_{n\in I} a_n \frac{\overline{F_{-n}}}{\sqrt{1+\left|\frac{2i\pi n}{L}\right|^{2s}}}   \sum_{p \in \Z} \frac{\varphi_p}{\lambda_{n+p}} e_p\right\|_{m}^2 \\
&=& \left\|  \sum_{p \in \Z} \varphi_p\sum_{n\in I} \frac{ a_n \overline{F_{-n}}}{\lambda_{n+p}\sqrt{1+\left|\frac{2i\pi n}{L}\right|^{2s}}} e_p\right\|_{m}^2 \\
&=&  \sum_{p \in \Z} \left(1+\left|\frac{2\pi p}{L} \right|^{2m}\right) |\varphi_p|^2\left|\sum_{n\in I} \frac{ a_n \overline{F_{-n}}}{\lambda_{n+p}\sqrt{1+\left|\frac{2i\pi n}{L}\right|^{2s}}} \right |^2.
\end{array}$$}
Now, using condition \eqref{phiGrowth}, we have
$$c^2 \sum_{p \in \Z}\left|\sum_{n\in I} \frac{ a_n \overline{F_{-n}}}{\lambda_{n+p}\sqrt{1+\left|\frac{2i\pi n}{L}\right|^{2s}}} \right|^2 
\leq \left\|\sum_{n\in I} a_n k_{n, \lambda}^s\right\|_{m}^2 
\leq C^2 \sum_{p \in \Z}\left|\sum_{n\in I} \frac{ a_n \overline{F_{-n}}}{\lambda_{n+p}\sqrt{1+\left|\frac{2i\pi n}{L}\right|^{2s}}} \right|^2, $$
where $c, C>0$ are the decay constants in condition \eqref{phiGrowth}.

This last inequality can be rewritten 
$$c^2 \left\|\sum_{n\in I} \frac{ a_n \overline{F_{-n}}}{\sqrt{1+\left|\frac{2i\pi n}{L}\right|^{2s}}}\Lambda^{\lambda}_n \right\|^2 
\leq \left\|\sum_{n\in I} a_n k_{n, \lambda}^s\right\|_{m}^2 
\leq C^2  \left\|\sum_{n\in I} \frac{ a_n \overline{F_{-n}}}{\sqrt{1+\left|\frac{2i\pi n}{L}\right|^{2s}}} \Lambda^{\lambda}_n \right\|^2 , $$
as 
$$\Lambda^{\lambda}_n=\sum_{p\in \Z} \frac1{\lambda_{n+p} }e_p.$$
We now use the fact that $(\Lambda^{\lambda}_n)$ is a Riesz basis of $L^2$: indeed, it is the image of the Hilbert basis $(e_n)$ by the isomorphism 
$$\Lambda^{\lambda} : f\in L^2 \mapsto{\color{black}\Lambda^\lambda_0} f.$$
The norms of $\Lambda^{\lambda}$ and its inverse are rather straightforward to compute using piecewise constant functions, we have
$$\vertiii {\Lambda^{\lambda}}  =\frac{\sqrt{L}}{1-e^{-\lambda L}},$$
$$\vertiii {(\Lambda^{\lambda})^{-1}}  = \frac{1-e^{-\lambda L}}{\sqrt{L}} e^{\lambda L},$$
so that
$$\frac{1}{\vertiii {(\Lambda^{\lambda})^{-1}} ^2}\sum_{n\in I} \left|\frac{ a_n \overline{F_{-n}}}{\sqrt{1+\left|\frac{2i\pi n}{L}\right|^{2s}}}\right|^2 \leq \left\|\sum_{n\in I} \frac{ a_n \overline{F_{-n}}}{\sqrt{1+\left|\frac{2i\pi n}{L}\right|^{2s}}}\Lambda^{\lambda}_n \right\|^2 \leq \vertiii {\Lambda^{\lambda}} ^2\sum_{n\in I} \left|\frac{ a_n \overline{F_{-n}}}{\sqrt{1+\left|\frac{2i\pi n}{L}\right|^{2s}}}\right|^2,$$
 and we finally get, using the fact that $(F_n)$ has $s$-growth,
$$c^2C_1^2 \frac{1}{\vertiii {(\Lambda^{\lambda})^{-1}} ^2}\sum_{n\in I} |a_n|^2 \leq \left\|\sum_{n\in I} a_n k_{n, \lambda}^s\right\|_{m}^2 \leq C^2C_2^2 \vertiii {\Lambda^{\lambda}} ^2\sum_{n\in I} |a_n|^2.$$
where $C_1, C_2 >0$ are the growth constants of $(F_n)$, so that the constants in the inequalities above are optimal. Hence, using again point 2. of Proposition \ref{Riesz}, $(k_{n, \lambda}^s)$ is a Riesz basis of $H^{m}_{per}$. 
\end{proof}

We now have candidates for the backstepping transformation, under some conditions on $F$:
\begin{crllr}\label{TInvert}
Let $m\in \NN^\ast$, and $F$ such that $(F_n)$ has $m$-growth, with growth constants $C_1, C_2>0$. Define
\begin{equation}\label{T}
\quad T^{\lambda} \alpha := \sum_{n\in \Z} \sqrt{1+\left|\frac{2i\pi n}{L}\right|^{2m}} \alpha_n k_{-n, \lambda}^{m}=\sum_{n\in \Z} \alpha_n k_{-n, \lambda} \in H^m_{per}, \quad \forall \alpha \in H^{m}_{per}, 
\end{equation}
where $\alpha=\sum_{n\in \Z} \alpha_n e_n$.
Then, $T^{\lambda}:H^{m}_{per} \rightarrow H^m_{per}$ is an isomorphism. Moreover,
\begin{equation}\label{Norms}
\begin{aligned}
\vertiii {T^{\lambda}} &{\color{black}\leq}  \frac{CC_2\sqrt{L}}{1-e^{-\lambda L}}, \\
\vertiii {(T^{\lambda})^{-1}} &{\color{black}\leq} \frac{1-e^{-\lambda L}}{cC_1  \sqrt{L}} e^{\lambda L}.
\end{aligned}
\end{equation}
\end{crllr}
\begin{proof}
The invertibility of $T^{\lambda}$ is clear thanks to the Riesz basis property of $( k_{-n, \lambda}^{m})$, and \eqref{Norms} comes from {\color{black}the above calculations}.
\end{proof}

\subsection{Definition of the feedback law}\label{sec-feedback-def}

In order to further determine the feedback law, and define our final candidate for the backstepping transformation, the idea is now to return to the $TB=B$ condition \eqref{TB}, as we have used it in the formal computations of section \ref{FindKernel}, in the equation \eqref{TB}. However, in this case, $\varphi \notin H^m_{per}$, and so it is not clear whether $T^\lambda \varphi$ is well-defined.

We can nonetheless obtain a $TB=B$ condition in some weak sense: indeed, let us set
$$\varphi^{(N)}:= \sum_{n=-N}^N \varphi_n e_n \in H^m_{per}, \quad \forall N \in \NN .$$
Then, 
$$ \varphi^{(N)} \xrightarrow[N \to \infty]{H^{m-1}} \varphi$$
and
$$\begin{array}{rcl}
T^{\lambda} \varphi^{(N)} &= &\sum_{n=-N}^N -\varphi_n \overline{F_{n}} \Lambda^{\lambda}_{-n} \star \varphi  \\
&=& \sum_{n=-N}^N \sum_{p \in \Z} \frac{-\varphi_n \overline{F_{n}}  \varphi_p}{\lambda_{-n+p}} e_p \\
&=& \sum_{p \in \Z}  \varphi_p \left(\sum_{n=-N}^N \frac{-\varphi_n \overline{F_{n}} }{\lambda_{-n+p}}\right) e_p.
\end{array}
$$
Now, notice that one can apply the Dirichlet convergence theorem for Fourier series (see for example \cite{Kahane}) to $\Lambda^{\lambda}_p, p \in \Z$ at $0$:
$$\sum_{n=-N}^N \frac{1 }{\lambda_{-n+p}} = \sum_{n=-N}^N \frac{1 }{\lambda_{n+p}}  \xrightarrow[N \to \infty]{} \sqrt{L}\frac{\Lambda^{\lambda}_p (0) + \Lambda^{\lambda}_p (L)}2 = \frac{L}{1-e^{-\lambda L}} \frac{1+e^{-\lambda L}}{2}. $$
Let us note
$$K(\lambda):=\frac2{L} \frac{1-e^{-\lambda L}}{1+e^{-\lambda L}},$$
and set
\begin{equation}\label{Feedback}
F_n:=-\frac{K(\lambda)}{\overline{\varphi_n}}, \quad \forall n \in \Z.
\end{equation}
This defines a feedback law $F\in \dist$ which is real-valued, as $\varphi$ is real-valued, and which has $m$-growth thanks to condition \eqref{phiGrowth}, so that $T^{\lambda}$ is a valid backstepping transformation. Moreover,
\begin{equation}
    \label{WeakTB}
     \langle T^{\lambda} \varphi^{(N)}, e_p \rangle = \varphi_p K(\lambda) \sum_{n=-N}^N \frac{1 }{\lambda_{-n+p}}\xrightarrow[N \to \infty]{} \varphi_p, \quad \forall p \in \Z,
\end{equation}
which corresponds to the $TB=B$ condition in some weak sense. 

With this feedback law, the backstepping transformation now writes
\begin{equation}\label{TFeedback}
 T^{\lambda} \alpha = \sum_{n\in \Z}  \alpha_n k_{-n, \lambda}, \quad \forall \alpha \in H^{m}_{per},
\end{equation}
and 
\begin{equation}\label{NormsFeedback}
\begin{aligned}
\vertiii {T^{\lambda}} &{\color{black}\leq}  \frac{CK(\lambda)\sqrt{L}}{c(1-e^{-\lambda L})}, \\
\vertiii {(T^{\lambda})^{-1}} &{\color{black}\leq}  \frac{C(1-e^{-\lambda L})}{cK(\lambda) \sqrt{L}} e^{\lambda L}.
\end{aligned}
\end{equation}

\subsection{Regularity of the feedback law}\label{section-feedback-reg}

Finally, in order to study the well-posedness of the closed-loop system corresponding to \eqref{Feedback}, we need some information on the regularity of $F$.

Let us first begin by a general lemma for linear forms with coefficients that have $m$-growth:
\begin{lmm}\label{Fcont}
Let $m \geq 0$, and $G\in \dist$ with $m$-growth.

Then, for all $s>1/2$, $G$ is defined on $H^{m+s}_{per}$, is continuous for $\|\cdot\|_{m+s}$, but not for $\norm_{m+\sigma}$, for $-m \leq \sigma < 1/2$.

In particular, the feedback law $F\in \dist$ defined by \eqref{Feedback} defines a linear form on $H^{m+1}_{per}$ which is continuous for $\norm_{m+1}$ but not for $\norm_{m}$.
\end{lmm}

\begin{proof}
Let $s > 1/2$, and let $\alpha\in H^{m+s}_{per}$. Using the growth conditions \eqref{FGrowth}, we can do the following computations for $ \alpha \in H^{m+s}_{per}$:
$$\begin{array}{rcl}
\sum_{n\in \Z} |G_n||\alpha_n| 
&\leq & C\sum_{n\in \Z} \sqrt{1+\left|\frac{2i \pi n}{L}\right|^{2m}}|\alpha_n| \\
&\leq & C^\prime \sum_{n\in \Z} \frac{1}{1+|n|^s}  \sqrt{1+\left|\frac{2i \pi n}{L}\right|^{2m+2s}} |\alpha_n|\\
&\leq & C^\prime \left(\sqrt{\sum_{n\in \Z} \frac{1}{(1+|n|^{s})^2}}\right) \|\alpha\|_{m +s}
\end{array}$$
where $C, C^\prime>0$ are constants that do not depend on $\alpha$, and where the last inequality is obtained using the Cauchy-Schwarz inequality. Thus $G$ is defined on $H^{m+s}_{per}$ by
$$\langle \alpha, G \rangle := \sum_{n\in \Z} G_n\alpha_n, \quad \forall \alpha \in H^{m+s}_{per},$$ 
and $G$ is continuous on $H^{m+s}_{per}$. 

On the other hand, let $-m \leq \sigma < 1/2$, and consider, for $N \geq 1$,
$$\gamma^{(N)}:= \sum_{|n|\geq N}\frac{1}{\overline{G_n}\left(1+|n|^{1+s}\right) } e_n \in H^{m +s}_{per}.$$ 
We have 
$$\|\gamma^{(N)} \|_{m + \sigma}^2 = \sum_{|n|\geq N} \frac{\left(1+\left|\frac{2i \pi n}{L}\right|^{2m+2\sigma}\right)}{|G_n|^2}\frac{1}{\left(1+|n|^{1 + s}\right)^2 } \leq C\sum_{|n|\geq N} \frac{1}{1+|n|^{2+2s-2\sigma} }$$
for some constant $C>0$.
Then, 
$$\begin{array}{rcl}
|\langle \gamma^{(N)}, G \rangle | & = & \sum_{|n|\geq N} \frac1{1+|n|^{1 + s}} \\
&\geq& c\sum_{|n| \geq N} |n|^{1+s-2\sigma} \frac1{1+|n|^{2+2s-2\sigma}} \\
&\geq & c N^{1+s-2\sigma} \sum_{|n| \geq N} \frac1{1+|n|^{2+2s-2\sigma}} \\
&\geq & c^\prime N^{1+s-2\sigma} \sqrt{\sum_{|n| \geq N} \frac1{1+|n|^{2+2s-2\sigma}} }\|\gamma^{(N)}\|_{m+\sigma}
\end{array}$$
for some constants $c, c^\prime >0$. Now, we know that there exists constants $c^{\prime\prime}, C^{\prime\prime} >0$ such that 
$$ \frac{c^{\prime\prime}}{N^{1+2s-2\sigma}}\leq \sum_{|n| \geq N} \frac1{1+|n|^{2+2s-2\sigma}}  \leq \frac{C^{\prime\prime}}{N^{1+2s-2\sigma}}, $$
So that 
$$N^{1+s-2\sigma}\sqrt{\sum_{|n| \geq N} \frac1{1+|n|^{2+2s-2\sigma}} } \geq c^{\prime\prime} N^{\frac12-\sigma} \xrightarrow[N \to \infty]{} \infty.$$
This proves that $G$ is not continuous for$\norm^{m + \sigma}$.
\end{proof}

\color{black} Let us now give a more precise description of the domain of definition and regularity of $F$. Recalling the identity \eqref{DAphi}, we can derive the following identity for $F_n$ from \eqref{Feedback}:
\begin{equation}\label{DAF} F_n =  (-1)^{m}\frac{K(\lambda)}{\tau^\varphi_{-n}} \left(\frac{2i \pi n}{L}\right)^m +(-1)^m \frac{K(\lambda)}{\tau^\varphi_{-n}} \left(\frac{2i \pi n}{L}\right)^m \frac{\overline{\displaystyle{\sum_{j=0}^{d}\left\langle \chi_{[\sigma_j, \sigma_{j+1}]}\partial^m \varphi, e_n\right\rangle}}}{\tau^\varphi_{-n}-\overline{\displaystyle{\sum_{j=0}^{d}\left\langle \chi_{[\sigma_j, \sigma_{j+1}]}\partial^m \varphi, e_n\right\rangle}}}, \quad \forall n \in \Z^\ast,
\end{equation}
so that
\begin{equation}\label{RegularPart}\left(\frac1{\left(\frac{2i\pi n}{L}\right)^m}\left(F_n -(-1)^{m} \frac{K(\lambda)}{\tau^\varphi_{-n}} \left(\frac{2i \pi n}{L}\right)^m  \right)\right)_{n\in \Z^\ast} \in \ell^2.\end{equation}

Let us then note 
$$ h_n:=(-1)^{m}\frac{K(\lambda)}{\tau^\varphi_{-n}} \left(\frac{2i \pi n}{L}\right)^m,\quad \forall n \in \Z, $$
and $h$ the associated linear form in $\dist$. 

\begin{prpstn}\label{FeedbackReg}
The linear form $h$ defines the following linear form on $\tau^\varphi(H^{m+1}_{(pw)})$, continuous for $\|\cdot\|_{m+1,pw}$: 
\begin{equation}\label{SingularPart} \langle \alpha, h \rangle=\sqrt{L}\frac{K(\lambda)}2 \left(\partial^m\left((\tau^\varphi)^{-1}\alpha\right)(0)+\partial^m\left((\tau^\varphi)^{-1}\alpha\right)(L)\right) , \quad \forall \alpha \in {\color{black}\tau^\varphi(H^{m+1}_{(pw)})}.\end{equation}
Moreover, $\tilde{F}:=F-h$ is continuous for $\|\cdot\|_m$, so that $F$ is defined on $\tau^\varphi(H^{m+1}_{(pw)})\cap H^m_{per}$, and is continuous for $\|\cdot\|_{m+1, pw}$, but not $\|\cdot\|_m$.
\end{prpstn}
\begin{proof}
It is clear, by definition of $H^m_{per}$, and using \eqref{RegularPart}, that for $\alpha \in H^m_{per}$, the expression:
\begin{equation}
    {\color{black}\langle \alpha, F-h \rangle = \sum_{n\in \Z} \alpha_n (\overline{F_n}-\overline{h_n}) = \frac{K(\lambda)\alpha_0}{\varphi_0} + \sum_{n \neq 0} \left(\frac{2i\pi n}{L}\right)^m\alpha_n  \frac1{\left(\frac{2i\pi n}{L}\right)^m}(\overline{F_n}{\color{black}-}\overline{h_n}) }
\end{equation}
defines a continuous linear form on $H^m_{per}$. 

On the other hand, let $\alpha \in \tau^\varphi (H^{m+2}_{(pw)})$, then
$$\sum_{n=-N}^N \alpha_n \overline{h_n} = \sqrt{L}K(\lambda)\sum_{n=-N}^N \left(\frac{2i \pi n}{L}\right)^m\frac{\alpha_n}{\tau^\varphi_n}\frac1{\sqrt{L}} $$
we can use the Dirichlet convergence theorem (see \cite{Kahane}) on $\partial^m \left((\tau^\varphi)^{-1}\alpha\right) \in H^2_{(pw)}$ at $0$, so that 
$$\begin{aligned}\sum_{n=-N}^N \alpha_n \overline{h_n} =  \sqrt{L}K(\lambda) &\sum_{n=-N}^N \left(\frac{2i \pi n}{L}\right)^m\frac{\alpha_n}{\tau^\varphi_n}\frac1{\sqrt{L}} \\ 
& \xrightarrow[N \to \infty]{}  \sqrt{L}\frac{K(\lambda)}2 \left( \partial^m\left((\tau^\varphi)^{-1}\alpha\right)(0)+\partial^m\left((\tau^\varphi)^{-1}\alpha\right)(L)\right).\end{aligned}$$
Now, we know that $H^{m+2}_{(pw)}$ is dense in $H^{m+1}_{(pw)}$ for the $H^{m+1}_{(pw)}$ norm. As $\tau^\varphi$ is a sum of translations, it is continuous for $\|\cdot\|_{m+1, pw}$, so that $\tau^\varphi(H^{m+2}_{(pw)})$ is dense in $\tau^\varphi(H^{m+1}_{(pw)})$ for $\|\cdot\|_{m+1, pw}$.

Moreover, using the Sobolev inequality for $H^1$ and $L^\infty$ (see for example \cite[Chapter 8, Theorem 8.8]{Brezis}), we get the continuity of $h$ for $\|\cdot\|_{m+1, pw}$, so that we can extend it from $\tau^\varphi(H^{m+2}_{(pw)})$ to $\tau^\varphi(H^{m+1}_{(pw)})$ by density. We also get that $h$ is not continuous for $\|\cdot\|_{m}$, as $\alpha \in H^m \mapsto \partial^m\alpha(0)$ and $\alpha \in H^m \mapsto \partial^m\alpha(L)$ are not continuous for $\|\cdot\|_{m}$.

Thus, $F=\tilde{F}+h$ is defined on $\tau^\varphi(H^{m+1}_{(pw)})\cap H^m_{per}$, is continuous for $\|\cdot\|_{m+1}$ but not for $\|\cdot\|_m$.
\end{proof}

\color{black}\section{Well-posedness and stability of the closed-loop system}
Let $m \geq 1$, $\varphi \in H^m_{(pw)} \cap H^{m-1}_{per}$ satisfying growth condition \eqref{phiGrowth}. Let the feedback law $F$ be defined by \eqref{Feedback}.
\subsection{Operator equality}

Now that we have completely defined the feedback $F$ and the transformation $T^\lambda$, let us check that we have indeed built a backstepping tranformation. 
As in the finite dimensional example of subsection \ref{FiniteD}, this corresponds to the formal operator equality
$$T(A+BK)=(A-\lambda I) T.$$

Let us define the following domain:
\begin{equation}\label{Domains}
\quad D_{m} := \left\{ \alpha \in \tau^\varphi(H^{m+1}_{(pw)})\cap H^m_{per}, \quad -\alpha_x - \mu \alpha +\langle \alpha, F \rangle \varphi \in H^m_{per}\right\}.
\end{equation}

Notice that, as $\varphi \in H^m_{(pw)}$, the condition $ \alpha \in H^{m+1}_{(pw)} \supset \tau^\varphi(H^{m+1}_{(pw)}) $ is necessary for $-\alpha_x - \mu \alpha +\langle \alpha, F \rangle \varphi$ to be in $H^m_{per}$ . 
Let us first check the following property:
\begin{prpstn}\label{Dense}
For $m \geq 1$, $D_{m}$ is dense in $H^{m}_{per}$ for $\norm_{m}$.
\end{prpstn}
\begin{proof}
It is clear that $H^{m+1}_{per} \subset \tau^\varphi\left(H^{m+1}_{(pw)}\right)$, so that
$$\mathcal{K}_m:=\left\{\alpha \in H^{m+1}_{per}, \ \langle \alpha, F \rangle=0 \right\} \subset D_m.$$
Now, by Lemma \ref{Fcont}, as $F$ has $m$-growth, $\mathcal{K}_m$ is dense in $H^{m +1}_{per}$ for $\norm_{m}$, as the kernel of the linear form $F$ which is not continuous for $\norm_{m}$. As $H^{m +1}_{per}$ is dense in $H^{m}_{per}$, then $D_{m}$ is dense in $H^{m}_{per}$ for $\norm_{m}$.  
\end{proof}

Now, on this dense domain, let us establish the operator equality:
\begin{prpstn}\label{OpEqProp} 
\begin{equation}\label{OpEq}
 T^{\lambda} (-\partial_x - \mu I + \langle  \cdot, F \rangle \varphi) \alpha = (-\partial_x - \lambda^\prime I)T^{\lambda} \alpha \quad \textrm{in} \ H^{m}_{per}, \quad \forall \alpha \in D_{m}.
\end{equation}
\end{prpstn}
\begin{proof}
First let us rewrite \eqref{OpEq} in terms of $\lambda$:
$$ T^{\lambda} (-\partial_x + \langle  \cdot, F \rangle \varphi) \alpha = (-\partial_x - \lambda I)T^{\lambda} \alpha \quad \textrm{in} \ H^{m}_{per}, \quad \forall \alpha \in D_{m}(F).$$

Let $\alpha \in D_{m}$. By definition of the domain $D_m$, the left-hand side of \eqref{OpEq} is a function of $H^m_{per} \subset \dist$, and by construction of $T^{\lambda}$, the right-hand side of \eqref{OpEq} is a function of $H^{m-1}_{per} \subset \dist$.  To prove that these functions are equal, it is thus sufficient to prove their equality in $\dist$. Let us then write each term of the equality against $e_n$ for $n\in \Z$. One has
 
{\def\arraystretch{3}$$\begin{array}{rcl}
\left\langle (-\partial_x - \lambda I)T^{\lambda} \alpha, e_n \right\rangle &=& \left\langle T^{\lambda}\alpha, \frac{2i \pi n}L e_n \right\rangle -\lambda \left\langle T^{\lambda} \alpha , e_n \right\rangle\\
&=& -\lambda_n \langle T^{\lambda}\alpha, e_n \rangle.\\

\end{array}$$}

Let us now prove that
\begin{equation}\label{weakOpEq}
 \langle T^{\lambda} (-\alpha_x + \langle \alpha, F \rangle \varphi), e_n\rangle = -\lambda_n \langle T^{\lambda}\alpha, e_n \rangle, \quad \forall n\in \Z.\end{equation}

Now, as we only have $\alpha_x \in H^{m-1}_{per}$, $T^{\lambda}\alpha_x$ is not defined \textit{a priori}. In order to allow for more computations, let us define
$$\begin{aligned}
 \alpha^{(N)}:= \sum_{n=-N}^N \alpha_n e_n,& \quad \forall N \in \NN , \\
\varphi^{(N)}:= \sum_{n=-N}^N \varphi_n e_n,& \end{aligned}$$
so that we have, by property of the partial Fourier sum of a $H^m_{per}$ function,
$$-\alpha^{(N)}_x + \langle \alpha, F \rangle \varphi^{(N)} \xrightarrow[N \to \infty]{H^m} -\alpha_x + \langle \alpha, F \rangle \varphi,$$ 
so that in particular,
\begin{equation}\label{weakOpEq1}\langle T^{\lambda} (-\alpha^{(N)}_x + \langle \alpha, F \rangle \varphi^{(N)}), e_n\rangle \xrightarrow[N \to \infty]{}  \langle T^{\lambda} (-\alpha_x + \langle \alpha, F \rangle \varphi), e_n\rangle \end{equation}

Let $N \in \NN$. We can now write
{\def\arraystretch{2}$$\begin{array}{rcl}
\langle T^{\lambda} (-\alpha^{(N)}_x + \langle \alpha, F \rangle \varphi^{(N)}), e_n\rangle & = & -\langle T^{\lambda} \alpha^{(N)}_x , e_n\rangle + \langle \alpha, F \rangle \langle T^{\lambda}\varphi^{(N)}, e_n\rangle \\
&=& -\left\langle \sum_{p=-N}^N \frac{2i\pi p}L \alpha_p k_{-p, \lambda}, e_n  \right\rangle + \langle \alpha, F \rangle \langle T^{\lambda}\varphi^{(N)}, e_n\rangle. 
\end{array}$$}

Now, using \eqref{kn1}, we get
$$\frac{2i\pi p}L k_{-p, \lambda} = (k_{-p, \lambda})_x + \lambda k_{-p, \lambda} + \overline{F_p}\varphi,$$
so that
$$-T^{\lambda} \alpha^{(N)}_x = \sum_{p=-N}^N \alpha_p \left((k_{-p, \lambda})_x  +\lambda k_{-p, \lambda} +\overline{F_p}\varphi\right).$$
Hence
$$-\langle T^{\lambda} \alpha^{(N)}_x , e_n\rangle = - \left\langle \left(T^{\lambda} \alpha^{(N)}\right)_x, e_n \right\rangle - \lambda \left\langle T^\lambda \alpha^{(N)}, e_n \right\rangle - \langle \alpha^{(N)}, F \rangle \varphi_n,$$
and finally,
\begin{equation}\label{weakOpEq2}
\begin{aligned}\langle T^{\lambda} (-\alpha^{(N)}_x + \langle \alpha, F \rangle \varphi^{(N)}), e_n\rangle =
-\lambda_n \left\langle T^{\lambda} \alpha^{(N)}, e_n \right\rangle  &+ \left( \langle \alpha - \alpha^{(N)} , F \rangle \right)\varphi_n \\
&+ \langle \alpha, F \rangle \left(\left\langle T^{\lambda}\varphi^{(N)}- \varphi, e_n \right\rangle \right).
\end{aligned}
\end{equation}
To deal with the third term of the right-hand side of this equality, recall that we have chosen a feedback law so that the $TB=B$ condition \eqref{TB} holds. Thus,
\begin{equation}\label{OpEqTerm1}\left\langle T^{\lambda}\varphi^{(N)}- \varphi, e_n \right\rangle \xrightarrow[N \to \infty]{} 0.\end{equation}
To deal with the second term, recall that $F$ is the sum of a regular part $\tilde{F}$ and a singular part $h$:
$$\langle \alpha - \alpha^{(N)} , F \rangle = \left\langle \alpha - \alpha^{(N)}, \tilde{F} \right\rangle +  \langle\alpha - \alpha^{(N)}, h \rangle.$$
Now, by definition of $\alpha^{(N)}$ and continuity of $\tilde{F}$ for $\norm_{m}$, 
\begin{equation}\label{OpEqTerm21}\left\langle \alpha - \alpha^{(N)}, \tilde{F} \right\rangle \xrightarrow[N\to \infty]{} 0.\end{equation}
\color{black}On the other hand, for all $N \in \NN$, 
{\def\arraystretch{3}\begin{equation}\label{hOtherForm}\begin{array}{rcl}
\langle \alpha^{(N)}, h \rangle &=& K(\lambda) \sum_{n=-N}^{N} \frac{\alpha_n}{\tau^\varphi_n} \left(\frac{2i \pi n}{L}\right)^m \\
&=& \frac{K(\lambda)}2 \sum_{n=-N}^{N} \left(\frac{\alpha_n}{\tau^\varphi_n} + (-1)^{m} \frac{\alpha_{-n}}{\tau^\varphi_{-n}}\right) \left(\frac{2i \pi n}{L}\right)^m. \\
&=&\sqrt{L}\frac{K(\lambda)}2 \partial^{m-1} \tilde{\tau}^\varphi \alpha^{(N)}_x (0) ,
\end{array}\end{equation}}
where
$$\tilde{\tau}^\varphi f = \sum_{n\in \Z} \left(\frac{f_n}{\tau^\varphi_n} + (-1)^{m-1} \frac{f_{-n}}{\tau^\varphi_{-n}}\right) e_n, \quad \forall f \in L^2.$$
Now, notice that, by definition of $\tau^\varphi$ and $D_m$, 
\begin{equation}\label{RegA+BK}
\tilde{\tau}^\varphi \left(-\alpha_x - \mu \alpha +\langle \alpha, F \rangle \varphi\right) \in H^m_{per}.\end{equation}

Moreover, using \eqref{DAphi}, we have for $n \in \Z^\ast$:
{\def\arraystretch{3}$$\begin{array}{rcl}
\frac{\varphi_n}{\tau^\varphi_n} + (-1)^{m-1} \frac{\varphi_{-n}}{\overline{\tau^\varphi_{-n}}} &=& \frac{\varphi_n}{\tau^\varphi_n} + (-1)^{m-1} \overline{\frac{\varphi_{n}}{\overline{\tau^\varphi_n}}} \\
&=& \frac{-1-(-1)^{m-1}(-1)^m }{\left(\frac{2i\pi}{L}n\right)^m} + \frac{r_n}{\left(\frac{2i\pi}{L}n\right)^m} \\
&=& \frac{r_n}{\left(\frac{2i\pi}{L}n\right)^m},
\end{array}$$}
where $r_n \in \ell^2$. Hence, $\tilde{\tau}^\varphi \varphi \in H^m_{per}$. This, together with \eqref{RegA+BK}, yields
$$\tilde{\tau}^\varphi \alpha_x  \in H^m_{per}.$$
This implies that 
$$\tilde{\tau}^\varphi \alpha^{(N)}_x \xrightarrow[N\to \infty]{H^{m}} \tilde{\tau}^\varphi\alpha_x,$$
as $\tilde{\tau}^\varphi \alpha^{(N)}_x$ is the partial sum of $\tilde{\tau}^\varphi \alpha_x$.

Hence, by continuity of $\alpha \mapsto \partial^{m-1} \alpha(0)$ for $\norm_{m}$, \eqref{hOtherForm} implies that
\begin{equation}\label{OpEqTerm22}\left\langle \alpha - \alpha^{(N)}, h \right\rangle \xrightarrow[N\to \infty]{} 0.\end{equation}

Finally, \eqref{weakOpEq2}, \eqref{OpEqTerm1}, \eqref{OpEqTerm21}, \eqref{OpEqTerm22}, and the continuity of $T^\lambda$ yield
$$\langle T^{\lambda} (-\alpha^{(N)}_x + \langle \alpha, F \rangle \varphi^{(N)}), e_n\rangle \xrightarrow[N\to \infty]{} -\lambda_n \left\langle T^\lambda \alpha, e_n \right\rangle.$$
This, put together with \eqref{weakOpEq1}, gives \eqref{weakOpEq} by uniqueness of the limit, which in turn proves \eqref{OpEq}.\color{black}

\end{proof}
 \begin{rmrk}
 When $\varphi \in H^m$, $\tau^\varphi$ is simply {\color{black}$(1/\sqrt{L})(\partial^{m-1} \varphi(L)-\partial^{m-1} \varphi(0)) Id$, $F$ is defined on $H^{m+1}\cap H^m_{per}$, and $\tilde{\tau}^\varphi \alpha$ is simply, up to a constant factor, the symmetrisation $\alpha + (-1)^{m-1} \alpha(L-\cdot)$}, which is $H^m_{per}$ if $\alpha\in H^m\cap H^{m-1}_{per}$.
 \end{rmrk}

\subsection{Well-posedness of the closed-loop system}

The operator equality we have established in the previous section means that $T^{\lambda}$ transforms, if they exist, solutions of the closed-loop system with a well-chosen feedback into solutions of the target system. Let us now check that the closed-loop system in question is indeed well-posed {\color{black}in some sense.

\begin{prpstn}
The operator $A+BK:=-\partial_x -\mu \alpha + \langle  \cdot, F \rangle \varphi$ defined on $D_{m}$ is a dense restriction of the infinitesimal generator of a $C^0$-semigroup on $H^{m}_{per}$.
\end{prpstn}

\begin{proof}
We know from Lemma \ref{Dense} that $A+BK$ is densely defined on $D_m \subset H^m_{per}$.

Now, define the following semigroup on $H^m_{per}$:
\begin{equation}\label{target-semigroup}
    S_{\lambda^\prime}(t)\alpha:=e^{-\lambda^\prime t}\alpha(\cdot-t), \quad \forall \alpha \in H^m_{per}, \quad t\geq 0,
\end{equation}
which corresponds to the target system \eqref{Target1}. Its infinitesimal is given by
\begin{equation}\label{target-infinitesimal-gen}
    \begin{array}{c}
    D^{\lambda^\prime}:=H^{m+1}_per, \\
    -\partial_x - \lambda^\prime I.
    \end{array}
\end{equation}
Now, define a second semigroup on $H^m_{per}$:
\begin{equation}\label{closed-loop-semi-group}
    S(t)\alpha:=(T^\lambda)^{-1} S_{\lambda^\prime}(t)T^\lambda \alpha, \quad \forall \alpha \in H^m_{per}, \quad t\geq 0.
\end{equation}
The infinitesimal generator of $S(t)$ is given, when it exists, by the limit of
\begin{equation}
    \frac{S(t)\alpha - \alpha}t = (T^\lambda)^{-1}\frac{S_{\lambda^\prime}(t)T^\lambda\alpha-T^\lambda\alpha}t,
\end{equation}
so, by \eqref{target-infinitesimal-gen}, the domain of the infinitesimal generator of $S(t)$ is $(T^\lambda)^{-1}(H^{m+1}_{per})$, and the infinitesimal generator itself is given by
\begin{equation}
    \frac{S(t)\alpha - \alpha}t \xrightarrow[t\to 0^+]{H^m} (T^\lambda)^{-1}(-\partial_x - \lambda^\prime I)T^\lambda\alpha.
\end{equation}
In particular, by \eqref{OpEq},
\begin{equation}
    (T^\lambda)^{-1}(-\partial_x - \lambda^\prime I)T^\lambda\alpha = (-\partial_x - \mu I +\langle \cdot, F \rangle \varphi)\alpha=(A+BK)\alpha,
\end{equation}
which proves the proposition.
\end{proof}

\subsection{Stability of the closed-loop system}

We can now prove Theorem \ref{MainResult}.

Let $S(t)$ the semigroup defined by \eqref{closed-loop-semi-group}, $\alpha \in H^m_{per}$.

By definition of $S(t)$, and using \eqref{NormsFeedback}, we then get, for $t \geq 0$,
{\def\arraystretch{1.5}$$\begin{array}{rcl}
 \| S(t)\alpha \|_{m} &\leq& \vertiii {(T^{\lambda})^{-1}}\|S_{\lambda^\prime}(t)T^{\lambda} \alpha \|_{m}  \\
&\leq & \vertiii{(T^{\lambda})^{-1}}\  e^{-\lambda^\prime t} \|T^{\lambda}\alpha\|_{m}\\
&\leq&  \vertiii{(T^{\lambda})^{-1}} \vertiii{T^{\lambda}} e^{-\lambda^\prime t} \|\alpha\|_{m} \\
&\leq &  \left(\frac{C}{c}\right)^2 e^{\lambda L }e^{-\lambda^\prime t} \|\alpha\|_{m} ,
\end{array}$$} 
which proves the exponential stability of the semigroup $S(t)$.

\ 

Now consider the particular case where $C=c>0$, and $\mu=0$ to simplify notations, together with: 
\begin{equation}
    \label{critical-ex-phi}
\varphi_n:=\frac{C}{\sqrt{1+\left|\frac{2i \pi n}L\right|^{2m}}}, \quad \forall n\in \Z,
\end{equation}
so that
\begin{equation}
    \label{critical-ex-norms}
   \|\alpha \star \varphi\|_m=C\|\alpha\|, \ \forall \alpha \in L^2, \quad \|\alpha \star F \|=\frac1C \|\alpha\|_m, \ \forall \alpha \in H^m_{per}.
\end{equation}
Now let $\varepsilon >0$. Keeping in mind that $(\chi_{[0,1/n]})_{n> 0}$ and $(\chi_{[L-1/n, L]})_{n> 0}$ are maximizing sequences for $\Lambda^\lambda$ and $\left(\Lambda^\lambda\right)^{-1}$ respectively, and using \eqref{T}, \eqref{Feedback}, \eqref{target-semigroup}, we get for $t_n:=L-1/n$:
\begin{equation}
\begin{aligned}
    S(t_n)(\chi_{[0,1/n]} \star \varphi )&= (T^{\lambda})^{-1}S_{\lambda^\prime}(t_n)T^{\lambda}(\chi_{[0,1/n]} \star \varphi ) \\
    &=-K(\lambda)(T^{\lambda})^{-1}S_{\lambda^\prime}\left(t_n\right) \left(\varphi \star \left(\Lambda^\lambda (\chi_{[0, 1/n]})\right)\right) \\
    &=-K(\lambda)\frac{e^{-\lambda t_n} \sqrt{L}}{1-e^{-\lambda L}}(T^{\lambda})^{-1} \varphi \star \left(\chi_{[L-1/n, L]}e^{-\lambda\left(\cdot-t_n\right)}\right) \\
    &=-e^{-\lambda t_n} e^{\lambda(L-1/n)}(\chi_{[L-1/n, L]} \star \varphi )\\
    &=-e^{-\lambda t_n} e^{\lambda(L-1/n)}(\chi_{[0, 1/n]} \star \varphi )(\cdot-t_n), \quad \forall n>0,
    \end{aligned}
\end{equation}
so that 
\begin{equation}
    \left\| S\left(t_n\right)(\chi_{[0,1/n]} \star \varphi )\right\|_m=e^{-\lambda t_n} e^{\lambda(L-1/n)} \|\chi_{[0,1/n]} \star \varphi\|_m, \quad \forall n>0.
\end{equation}
Then, there exists $n>0$ such that 
\begin{equation}
     \left\| S\left(t_n\right)(\chi_{[0,1/n]} \star \varphi )\right\|_m>e^{-\lambda t_n} (e^{\lambda L}-\varepsilon)\|\chi_{[0,1/n]} \star \varphi\|_m.
\end{equation}
This shows that estimate \eqref{exp-stab-estimate} can be critical in some cases.
}

\subsection{Application}\label{section-application}
Let $m=1,\ \lambda >0$, and let us suppose, to simplify the computations, that $a\equiv 0$. Define
\begin{equation}\label{phiEx}
 \varphi(x)=L-x,  \quad \forall x \in (0, L),
\end{equation}
so that $\varphi\in H^1$ but is not periodic, and satisfies \eqref{phiGrowth}, with
$$\begin{aligned}
 \varphi_n=-\frac{i L^{\frac32}}{2\pi n},& \quad \forall n \in \Z^\ast, \\
 \varphi_0= \frac{L^{\frac32}}2.&
\end{aligned}$$
Then, 
\begin{equation}\label{Fex}
 \langle \alpha, F \rangle = -\frac{2K(\lambda)}{L^{\frac32}} \alpha_0 - K(\lambda) \frac{\alpha_x(0)+\alpha_x(L)}{2}, \quad \forall \alpha \in H^{2} \cap H^1_{per},
\end{equation}
and 
$$D_1=\left\{\alpha \in H^2 \cap H^1_{per}, \quad \frac{2K(\lambda)}{L^{\frac32}} \alpha_0 + \left(\frac1L - K(\lambda)\right)\alpha_x(0)-\frac1L \alpha_x (L)=0 \right\},$$
so that 
\begin{equation}\label{Systemex}
\left\{\begin{aligned}
\alpha_t + \alpha_x &=  \left(-\frac{2K(\lambda)}{L^{\frac32}} \alpha_0 - K(\lambda) \frac{\alpha_x(0)+\alpha_x(L)}{2} \right)\varphi (x), \ x\in [0,L], \\
\alpha(t,0) & = \alpha(t,L), \ \forall t \geq 0,
\end{aligned}\right.
\end{equation}

\noindent has a unique solution for initial conditions in $D_1$.

The backstepping transformation can be written as:
\begin{equation}\label{Tex}
 T^\lambda\alpha = \frac{\sqrt{L}}{1-e^{-\lambda L}}\left(e^{-\lambda x}\left(-\frac{K(\lambda)}{\sqrt{L}} \alpha_x -  \frac{2K(\lambda)}{L^2} \alpha_0\right)\right) \star \varphi , \quad \forall \alpha \in H^1_{per}.
\end{equation}
Let $\alpha(t) \in D_1$ be the solution of the closed loop system \eqref{Systemex} with initial condition $\alpha^0 \in D_1$, and let us note $z(t):=T^\lambda \alpha(t)$, then
{\def\arraystretch{2}$$\begin{array}{rcl}
z_t &=& \frac{\sqrt{L}}{1-e^{-\lambda L}}\left(e^{-\lambda x}\left(-\frac{K(\lambda)}{\sqrt{L}} \alpha_{xt} -  \frac{2K(\lambda)}{L^2} \alpha_0^\prime\right)\right) \star \varphi. \\
&=& \frac{\sqrt{L}}{1-e^{-\lambda L}}\left(e^{-\lambda x}\left(-\frac{K(\lambda)}{\sqrt{L}} (-\alpha_{xx}+ \langle \alpha, F \rangle \varphi_x) -  \frac{2K(\lambda)}{L^2} \alpha_0^\prime\right)\right) \star \varphi. \\
&=& \frac{\sqrt{L}}{1-e^{-\lambda L}}\left(e^{-\lambda x}\left(-\frac{K(\lambda)}{\sqrt{L}} (-\alpha_{xx}- \langle \alpha, F \rangle) -  \frac{2K(\lambda)}{L^2} \alpha_0^\prime\right)\right) \star \varphi. \\
z_x&=& \frac{\sqrt{L}}{1-e^{-\lambda L}}\left(-e^{-\lambda x}\frac{K(\lambda)}{\sqrt{L}} \alpha_{xx} \right) \star \varphi - \lambda z \\
z_t+z_x + \lambda z &=& \frac{\sqrt{L}}{1-e^{-\lambda L}}\left(e^{-\lambda x}\left(\frac{K(\lambda)}{\sqrt{L}} \langle \alpha, F \rangle -  \frac{2K(\lambda)}{L^2} \alpha_0^\prime\right)\right) \star \varphi.
\end{array}$$}
By projecting the closed loop system on $e_0$, we get
$$\alpha_0^\prime = \langle \alpha, F\rangle \varphi_0 = \langle \alpha, F \rangle \frac{L^{\frac32}}2$$
so that
$$z_t+z_x + \lambda z =0.$$
In particular, 
\begin{equation}\label{energyDissip}
\frac{d}{dt}\|z\|^2_1=-2\lambda\|z\|^2_1.\end{equation}
Let us now set
$$ V(\alpha):= \|z\|^2_1, \quad \forall \alpha \in H^1_{per}.$$

Now, notice that
{\def\arraystretch{3}$$\begin{array}{rcl}
\|T^\lambda \alpha \|^2_1&=& \frac{L}{(1-e^{-\lambda L})^2} \sum_{n\in \Z} \left(1+\left|\frac{2i\pi n}{L}\right|^2\right) |\varphi_n|^2 \left|\left\langle e^{-\lambda x}\left(-\frac{K(\lambda)}{\sqrt{L}} \alpha_x -  \frac{2K(\lambda)}{L^2} \alpha_0\right), e_n \right\rangle\right|^2 \\
&\geq& C  \left\| e^{-\lambda x}\left(-\frac{K(\lambda)}{\sqrt{L}} \alpha_x -  \frac{2K(\lambda)}{L^2} \alpha_0\right) \right\|^2 \\
&\geq& C e^{2\lambda L} \left\| -\frac{K(\lambda)}{\sqrt{L}} \alpha_x -  \frac{2K(\lambda)}{L^2} \alpha_0 \right\|^2 \\
&\geq& C^\prime K(\lambda)^2 e^{2\lambda L} \|\alpha\|_1^2.
\end{array}$$}
Together with \eqref{energyDissip}, this shows that $V$ is a Lyapunov function, and \eqref{Systemex} is exponentially stable.
\section{Further remarks and questions}

\subsection{Controllability and the $TB=B$ condition}
\

In the introduction we have mentioned that the growth constraint on the Fourier coefficients of $\varphi$ actually corresponds to the exact null controllability condition in some Sobolev space for the control system \eqref{System}. As we have mentioned in the finite dimensional example, the controllability condition is essential to solve the operator equation: in our case, formal computations lead to a family of functions that turns out to be a Riesz basis precisely thanks to that rate of growth. Moreover, that rate of growth is essential for the compatibility of the $TB=B$ condition and the invertibility of the backstepping transformation. Indeed, as the transformation is constructed formally using a formal $TB=B$ condition, that same $TB=B$ condition fixes the value of the coefficients of $F_n$, giving them the right rate of growth for $T^\lambda$ to be an isomorphism.

In that spirit, it would be interesting to investigate if a backstepping approach is still valid if the conditions on $\varphi$ are weakened. For example, if we suppose approximate controllability instead of exact controllability, i.e.
$$\varphi_n \neq 0, \quad \forall n \in \Z ,$$
then $F$ can still be defined using a weak $TB=B$ condition. However, it seems delicate to prove, in the same direct way as we have done, that $T^\lambda$ is an isomorphism, as we only get the completeness of the corresponding $\left(k_{n, \lambda}\right)$, but not the Riesz basis property. 

Finally, it should be noted that, while in \cite{Schrodinger} the $TB=B$ condition is well-defined, in our case, it only holds in a rather weak sense. This is probably because of a lack of regularization, indeed in \cite{Schrodinger} the backstepping transformation has nice properties, as it can be decomposed in Fredholm form, i.e. as the sum of a isomorphism and a compact operator. Accordingly, the Riesz basis in that case is quadratically close to the orthonormal basis given by the eigenvectors of the Laplacian operator. That is not the case for our backstepping transformation, as it is closely linked to the operator $\Lambda^\lambda$, which does not have any nice spectral properties.

Nonetheless, it appears that thanks to some information on the regularity of $F$, a weak sense is sufficient and allows us to prove the operator equality by convergence.

\subsection{Regularity of the feedback law}

As we have pointed out in Section \ref{section-feedback-reg}, if $\varphi$ is such that system \eqref{System} is controllable in $H^m_{per}$, then the feedback law $F$ defined by \eqref{Feedback} is continuous for $\norm_{m+1}$ but not for $\norm_{m}$. This was actually to be expected, as we have mentioned in the introduction that Shun Hua Sun proved that bounded feedback laws can only achieve ``compact'' perturbations of the spectrum, which is not enough to get exponential stabilization. More precisely, it would be possible to get exponential stabilization only with very singular controllers. With a distributed control such as ours, it is necessary to consider unbounded feedback laws.

Moreover, the application in Section \ref{section-application} shows that even though the feedback is not continuous, and is given by its Fourier coefficients, in practice it can be expressed quite simply for some controllers. 

\subsection{Null-controllability and finite-time stabilization}
\

As we have mentioned in the introduction, one of the advantages of the backstepping method is that it can provide an explicit expression for feedbacks, thus allowing the construction of explicit controls for null controllability, as well as time-varying feedbacks that stabilize the system in finite time $T>0$. 

The general strategy (as is done in \cite{CoronNguyen}, \cite{XiangKdVNull}) is to divide the interval $[0,T]$ in smaller intervals $[t_n, t_{n+1}]$, the length of which tends to $0$, and on which one applies feedbacks to get exponential stabilization with decay rates $\lambda_n$, with $\lambda_n \rightarrow \infty$. Then, for well-chosen $t_n, \lambda_n$, the trajectory thus obtained reaches $0$ in time $T$. Though this provides an explicit control to steer the system to $0$, the norm of the operators applied successively to obtain the control tends to infinity. As such, it does not provide a reasonably regular feedback. However, the previous construction of the control can be used, with some adequate modifications (see \cite{CoronNguyen} and \cite{XiangKdVFinite}) to design a time-varying, periodic feedback, with some regularity in the state variable, which stabilizes the system in finite time.

Let us first note that, due to the hyperbolic nature of the system, there is a minimal control time, and thus small-time stabilization cannot be expected. Moreover, even for $T>L$, the estimates we have established on the backstepping transforms prevent us from applying the strategy we have described above: indeed, for any sequences $(t_n) \rightarrow T$, $\lambda_n \rightarrow \infty$, we have
$$\|\alpha(t)\|_{m} \leq \prod_{k=0}^n \left(\frac{C}{c}\right)^{2n} e^{n\mu L} exp\left(\sum_{k=0}^n -\lambda_k(t_{k+1}-t_k -L)\right) \|\alpha_0\|_{m}, \quad \forall t \in [t_n, t_{n+1}], $$
where $c, C$ are the decay constants in \eqref{phiGrowth}.
Moreover, as $t_{k+1}-t_k \rightarrow 0$, we have
$$exp\left(\sum_{k=0}^n -\lambda_k(t_{k+1}-t_k -L)\right) \xrightarrow[n \to \infty]{} \infty.$$

Another approach could be to draw from \cite{CHO2} and apply a second transformation to design a more efficient feedback law. It would also be interesting to adapt the strategy in \cite{YuXuJiangGanesan}, inspired from \cite{KrsticHaraTsub}, to our setting.
\subsection{Nonlinear systems}
\

Finally, another prospect, having obtained explicit feedbacks that stabilize the linear system, is to investigate the stabilization of nonlinear transport equations. This has been done in \cite{CoronLu2}, where the authors show that the feedback law obtained for the linear Korteweg-de Vries equation also stabilize the nonlinear equation. However, as in \cite{Schrodinger}, the feedback law we have obtained is not continuous in the norm for which the system is stabilize. This would require some nonlinear modifications to the feedback law in order to stabilize the nonlinear system.

\section*{Acknowledgements}
This work is partially supported by ANR project Finite4SoS
(ANR-15-CE23-0007), and the French Corps des Mines. 

The author would like to thank Jean-Michel Coron, for bringing the problem to his attention, his constant support, and his valuable remarks, as well as Amaury Hayat and Shengquan Xiang for discussions on this problem.

\end{document}